\documentclass[11pt, oneside]{amsart}

\usepackage{tabularx, hyperref}
\usepackage{amssymb} \usepackage{amsfonts} \usepackage{amsmath}
\usepackage{amsthm} \usepackage{epsfig, subfig}
\usepackage{ amscd, amsxtra, latexsym}
\usepackage[all]{xy}
\usepackage{caption}
\usepackage{enumerate}
\usepackage{color}

\newtheorem{lemma}{Lemma}[section]
\newtheorem{thm}[lemma]{Theorem}

\newtheorem{cor}[lemma]{Corollary}

\theoremstyle{definition}
\newtheorem{defn}[lemma]{Definition}

\theoremstyle{definition}

\definecolor{darkgreen}{cmyk}{1,0,1,.2}

\newcommand{\calC} {\ensuremath {\mathcal{C}}}

\newcommand{\calH} {\ensuremath {\mathcal{H}}}
\newcommand{\calD} {\ensuremath {\mathcal{D}}}

\address{Department of Mathematics, ETH Zurich, 8092 Zurich, Switzerland}
\email{sisto@math.ethz.ch}

\begin{document}

\title{Many Haken Heegaard splittings}
\author{Alessandro Sisto}

 \begin{abstract}
We give a simple criterion for a Heegaard splitting to yield a Haken manifold. As a consequence, we construct many Haken manifolds, in particular homology spheres, with prescribed properties, namely Heegaard genus, Heegaard distance and Casson invariant.

Along the way we give simpler and shorter proofs of the existence of splittings with specified Heegaard distance, originally proven by Ido-Jang-Kobayashi, of the existence of hyperbolic manifolds with prescribed Casson invariant, originally due to Lubotzky-Maher-Wu, and of a result about subsurface projections of disc sets (for which we even get better constants), originally due to Masur-Schleimer.
 \end{abstract}

\maketitle

\section{Introduction}
Every closed connected oriented $3$--manifold admits a Heegaard splitting, meaning that it can be obtained by gluing two handlebodies along their boundaries. It is therefore interesting to understand what kind of information one can extract about the $3$--manifold from the gluing map that defines one of its Heegaard splittings. For example, in a seminal work Hempel \cite{Hempel:Heegaard} proved that if a $3$--manifold is Seifert fibered or contains an incompressible torus then, for any Heegaard splitting, the distance between the disc sets of the handlebodies in the curve graph of their common boundary, which we will call Heegaard distance, is at most $2$. In particular, because of geometrisation, if a $3$--manifold admits a Heegaard spitting of distance at least $3$ then it is hyperbolic.

In this paper, we study how the property of being Haken can be read off the data coming from a Heegaard splitting.

 (Recall that the surface $S$ embedded in the $3$--manifold $M$ admits a \emph{compression disc} if there exists an embedding $f$ of the $2$--disc $D^2$ into $M$ so that $f(\partial D^2)$ is an essential loop in $S$ and $f(\mathring{D^2})\subseteq M-S$. A closed connected oriented $3$--manifold is \emph{Haken} if it is irreducible and it contains an incompressible surface, i.e. an embedded connected orientable surface not homeomorphic to the $2$--sphere that does not admit a compression disc.)
 
 We give a simple criterion for a Heegaard splitting to yield a Haken manifold (Theorem \ref{thm:steady_incompressible}). Roughly speaking, the criterion applies when there is a tight geodesic connecting the disc sets along which there are large subsurface projections at every point. We will provide an explicit construction of an embedded surface (see Definition \ref{defn:steadysurface}) starting from a path in the curve graph with certain properties (specified in Definition \ref{defn:steadypath}), and then prove that such surface does not admit compression discs.
 
 It is relatively easy to construct splittings so that the criterion applies, and in fact there is a lot of flexibility in doing so. In particular, we can construct Haken manifolds, and in particular Haken homology spheres, that satisfy a rather long list of prescribed properties:
 
% 
% The Heegaard distance of a Heegaard splitting is the distance in the curve graph of the Heegaard surface between the two disc sets coming from the handlebodies, see Subsection \ref{subsec:heegard}.
% 
 \begin{thm}\label{thm:manyhaken}
  Let $g$, $n$ and $k$ be integers and suppose that either $g,n\geq 3$. Then there exists a closed oriented $3$--manifold with the following properties:
  \begin{itemize}
   \item $M$ is Haken,
   \item $M$ is an integer homology sphere,
   \item $M$ is hyperbolic,
   \item $M$ has a Heegaard splitting of genus $g$ and Heegaard distance $n$,
   \item $M$ has Casson invariant $k$.   
  \end{itemize}
 \end{thm}
 
 We emphasize that various subsets of those properties were not known to be simultaneously realisable. In fact, for example, the first construction of splittings of given Heegaard distance is given in \cite{Heegaard_distance_n}, but these are not guaranteed to be neither Haken nor homology spheres. (There is a construction of Haken manifolds with splittings of arbitrarily large distance \cite{Heegaard_distance_at_least_n}, see also \cite{high_distance_knots}, but such manifolds have positive first Betti number.)  Also, the only previously known construction of hyperbolic manifolds with given Casson invariant is the one in \cite{LubotzkyMaherWu}, where the authors do not obtain precise control on the Heegaard distance and do not show whether their manifolds are Haken or not. In fact, our construction is shorter and simpler than the ones in either of these papers, especially \cite{LubotzkyMaherWu}, which uses probabilistic methods.
 
 Along the way, see Corollary \ref{cor:MasurSchleimer}, we also improve the bounds and give a simpler proof of a useful result from \cite{MasurSchleimer} about subsurface projections of disc sets.

 \subsection*{Acknowledgement} This paper would have not been possible without the contribution of Saul Schleimer, who provided precious suggestions and insights about the construction of the mapping classes in Section \ref{sec:constr}.
 
 The author would also like to especially thank Jeff Brock for the discussions that lead to the idea for constructing incompressible surfaces, as well as Sebastian Hensel, Joseph Maher, Kasra Rafi and Juan Souto for interesting discussions.
 
 This material is based upon work supported by the National Science Foundation under grant No. DMS-1440140 while the author was in residence at the Mathematical Sciences Research Institute in Berkeley, California, during the Fall 2016 semester.

\section{Steady paths}

\subsection{Background and conventions}
We denote by $\Sigma_g$ the closed connected oriented surface of genus $g$, and we will always assume $g\geq 2$.

%\subsection{Curve graph}
For short, we will write ``curve'' instead of ``essential simple closed curve'' (recall that a curve is essential if it does not bound a disc or a once-punctured disc). The curve graph $\mathcal C(\Sigma_g)$ of $\Sigma_g$ is the graph whose vertices are isotopy classes of curves on $\Sigma_g$ and where two curves are connected by an edge if and only if they have disjoint representatives. With an abuse of notation, when discussing properties of a set of curves we will implicitly assume that they are in minimal position, and when referring to the distance between two curves in $\mathcal C(\Sigma_g)$ we will mean the distance between their isotopy classes. 

We say that a subsurface of $\Sigma_g$ is non-sporadic if it is not a sphere with at most 4 discs removed or a torus with one disc removed. Similarly to the above, given an essential non-sporadic subsurface $Y$ of $\Sigma_g$ we denote by $\mathcal{AC}(Y)$ the graph whose vertices are isotopy classes of curves and essential simple arcs in $Y$ (an arc is essential if it is does not cut out a disc), with distance defined as above. For sporadic surfaces the definitions need to be adjusted; we do not recall them here since this does not play a big role in this paper, and we refer the reader to \cite{MasurMinsky:I}.
%For short, when this happens we say that the curves are disjoint, and similarly we say that two curves intersect if all representatives in the respective isotopy classes intersect.

Curve graphs, as well as arc and curve graphs, are Gromov-hyperbolic \cite{MasurMinsky:I}, see also \cite{Aougab:uniform,Bowditch:uniform,ClayRafiSchleimer:uniform,HenselPrzytyckiWebb:unicorn,PrzytyckiSisto:universe}. (This fact only plays a minor role in this paper.)

A multicurve is a collection of disjoint pairwise non-isotopic curves. Given a multicurve $c$, we denote $N(c)$ an open regular neighborhood.

\subsubsection{Subsurface projections} We now recall some properties of subsurface projections. The statement and proof of the criterion for being Haken do not rely on this notion, but the construction of disc sets where the criterion applies does. To the best of the author's knowledge, the interaction between subsurface projections and Heegaard splittings was first studied in \cite{JMM:subsurface_proj_Heegard}.

For $Y$ an essential non-sporadic subsurface of $\Sigma_g$ and a curve $c$ in $\Sigma_g$, the subsurface projection $\pi_Y(c)\subseteq \mathcal{AC}(Y)$ is obtained as follows. First, one isotopes $c$ so that it intersects $\partial Y$ minimally. Then, one considers all connected components of $c\cap Y$, and defines $\pi_Y(c)$ as the set of all isotopy classes of arcs and curves that they represent. This is a set of diameter at most $1$ in $\mathcal{AC}(Y)$. We will write $d_{\mathcal{AC}(Y)}(c,c')$ for $d_{\mathcal{AC}(Y)}(\pi_Y(c),\pi_Y(c'))$.

One of the fundamental facts about surface projection, and one that we will use repeatedly, is the Bounded Geodesic Image Theorem:

\begin{thm}[Bounded Geodesic Image Theorem, \cite{MasurMinsky:II}, see also \cite{Webb:BGI}]
There exists $C\geq 0$ with the following property. Let $Y\subseteq Z$ be essential subsurfaces of $\Sigma_g$. If $c_0,\dots,c_n$ are curves that form a geodesic in $\mathcal{AC}(Z)$ and $\pi_Y(c_i)$ is non-empty for every $Y$, then $d_{\mathcal AC(Y)}(c_0,c_n)\leq C$. 
\end{thm}

%(The theorem also holds for sporadic subsurfaces, even though we have not defined their arc and curve graphs. We only use this, for annuli, in Lemma \ref{lem:product_twists}.)

The following is a well-known easy consequence of the Bounded Geodesic Image Theorem.

\begin{lemma}\label{lem:concat_large_angle}
 There exists $K$ so that whenever $c_0,\dots,c_n$ is a sequence of curves in an essential subsurface $Y$ of $\Sigma_g$ where consecutive curves are disjoint and not isotopic, and $d_{\mathcal{AC}(\Sigma_g-N(c_i))}(c_{i-1},c_{i+1})\geq K$ for all $i=1,\dots, n-1$, then any geodesic in $\mathcal{AC}(Y)$ from $c_0$ to $c_n$ contains all the $c_i$.
\end{lemma}

\begin{proof}
We let $K=2C+1$, for $C$ as in the Bounded Geodesic Image Theorem.

 We argue by contradiction. Suppose that $i\geq 2$ is minimal so that some geodesic $\gamma$ from $c_0$ to $c_i$ does not contain $c_{i-1}$ (for $i=1$ the statement is obvious, and by minimality we only have to show that $\gamma$ contains $c_{i-1}$). Then by the Bounded Geodesic Image Theorem we have $d_{\mathcal{AC}(\Sigma_g-N(c_i))}(c_{0},c_{i+1})\leq C$. However, we can also apply the same theorem to $c_0,\dots,c_{i-1}$ and get $d_{\mathcal{AC}(\Sigma_g-N(c_i))}(c_{0},c_{i-1})\leq C$. But then we would have $d_{\mathcal{AC}(\Sigma_g-N(c_i))}(c_{i-1},c_{i+1})\leq 2C$, a contradiction. 
\end{proof}

\subsection{Definition of steady paths}

We will construct surfaces in Heegaard splittings starting from paths (of multicurves) in the curve graph with certain properties described below. The key condition is a large links condition, item \ref{item:large_projection}.

The notion of steady path we describe below is related to the notion of tight geodesics as defined in \cite{MasurMinsky:I}, and in particular for $d$ large enough a steady path is a tight geodesic (this fact does not get used in the proof of the criterion for being Haken).

\begin{defn}\label{defn:steadypath}
 For $\calD_0,\calD_1$ two sets of curves on $\Sigma_g$, where $g\geq 2$, we say that a sequence of multicurves $t_0,\dots,t_n$ is a $(\calD_0,\calD_1,d)$--\emph{steady path} if
\begin{enumerate}
\item the curves of the multicurve $t_0$ (resp. $t_n$) are in $\calD_0$ (resp. $\calD_1$),\label{item:endpoints}
\item whenever $c\in t_i,c'\in t_{i+1}$, we have that $d_{\mathcal C(\Sigma_g)}(c,c')=1$,\label{item:path}
%\item $d(a,b)=|i-j|$ whenever $a$ is a curve in $t_i$, $b$ is a curve in $t_j$ and $i\neq j$,
\item for every $i\neq 0,1$ (resp. $i\neq n-1,n$), every $d\in \calD_0$ (resp. $d\in\calD_1$) intersects $t_i$,\label{item:filling}
\item every $d\in \calD_0$ (resp. $d\in\calD_1$) not in $t_0$ (resp. $t_n$) intersects $t_0\cup t_1$ (resp. $t_{n-1}\cup t_n$),\label{item:maximal}
 \item $d_{\mathcal{AC}(\Sigma_g-N(t_i))}(t_{i-1},t_{i+1})\geq d$ for all $i=1,\dots,n-1$.\label{item:large_projection}
 %, meaning that any essential arc in $\Sigma_g-N(t_i)$ intersects $t_{i-1}\cup t_{i+1}$ at least $n$ times.
  %are no essential curves in $\Sigma_g-t_i$ disjoint from both $t_{i-1}$ and $t_{i+1}$.
\end{enumerate}

\end{defn}

\subsection{Concatenations of arcs}

The following lemma will be important to rule out compression discs. Essentially, the loop $\ell$ describes the shape of the boundary of a disc that we will encounter later in an argument by contradiction. In that context, the first condition will be guaranteed by Definition \ref{defn:steadypath}.\ref{item:large_projection}.

\begin{lemma}\label{lem:concat_fill_arcs}
% Let $t,u,v$ be disjoint curves in $\Sigma_g$, $g\geq 2$. Let $N(u)$ be an open regular neighborhood of $u$ disjoint from $t,v$.
%  If $d_{\mathcal{AC}(\Sigma_g-N(u))}(t,v)\geq 5$ then there does not exist a homotopically trivial loop $\ell$ obtained concatenating essential arcs $\alpha_1,\dots,\alpha_k$ in  $\Sigma_g\setminus N(u)$ such that:
%  \begin{itemize}
%   \item If $i$ is even then $\alpha_i$ is disjoint from from $t$.
%   \item If $i$ is odd then $\alpha_i$ intersects $v$ at most once.
%   \item $\alpha_i$ is disjoint from $\alpha_j$ when $i\cong j (2)$ and $i\neq j$. 
% %  \item distinct $\alpha_i$ are in minimal position.
%  \end{itemize}
Let $Y$ be a compact surface with boundary. Then there does not exist a homotopically trivial loop $\ell$ obtained by concatenating essential arcs $\alpha_1,\dots,\alpha_k$ in  $Y$ such that:
  \begin{itemize}
   \item If $i$ is even and $j$ is odd then $d_{\mathcal{AC}(Y)}(\alpha_i,\alpha_j)> 1$ (i.e., the arcs intersect essentially),
%   \item If $i$ is odd then $\alpha_i$ intersects $v$ at most once.
   \item $\alpha_i$ is disjoint from $\alpha_j$ when $i\cong j (2)$ and $i\neq j$. 
 %  \item distinct $\alpha_i$ are in minimal position.
  \end{itemize}
\end{lemma}

\begin{proof}
Consider a loop $\ell$ obtained concatenating arcs as in the statement. We can assume that the $\alpha_i$ are in minimal position relative to their endpoints, i.e. that they do not form bigons.
% and that they

  Consider a lift $\tilde \ell$ of $\ell$ in the universal cover $\tilde Y$ of $Y$. The goal is to show that $\tilde\ell$ is not a loop. Denote by $\tilde\alpha_i$ the lifts of the $\alpha_i$ that concatenate to form $\tilde \ell$. 
  %We can assume that the $\tilde\alpha_i$ are all distinct (otherwise there exists a shorter loop).
 
 The key facts that we will use are that
 \begin{enumerate}
  \item each lift of an $\alpha_i$ separates $\tilde Y$,
  \item $\tilde\alpha_i$ intersects some lift of $\alpha_j$ in its interior if and only if $i\not\cong j (2)$, and
  \item if a lift of $\alpha_i$ intersects a lift of $\alpha_j$, then it does so in one point.
 \end{enumerate}

%The second item follows from the distance condition in the arc and curve complex of $\Sigma_g-u$. Namely, $\alpha_i$ and $\alpha_j$, say for $i$ even and $j$ odd, must intersect because $\alpha_i$ is at distance at most $1$ from $t$ and $\alpha_j$ is at distance at most $2$ from $v$. 
(The third item follows from minimal position.)
 
 Notice that the number of arcs $k$ is at least $2$. Let $n$ be even so that the interior of $\alpha_n$ intersects $\bigcup_{m\, odd}\alpha_m$ in the minimal number of points among all even $n$. For $i$ odd, let $\tilde \alpha^i_n$ be a lift of $\alpha_n$ that intersects $\tilde\alpha_i$ in its interior. We are going to show that the $\tilde\alpha^i_n$ are all disjoint. This implies that the endpoints of $\tilde\ell$ are on opposite sides of, say, $\tilde\alpha^1_n$ (since traveling along $\tilde\ell$ one crosses all the $\tilde\alpha^i_n$ exactly once). Hence, this ensures that $\tilde\ell$ is not a loop.
 
 Suppose by contradiction that two $\tilde\alpha^i_n$ coincide (notice that two lifts of $\alpha_n$ are disjoint if and only if they do not coincide). We can then consider distinct odd indices $i,j$ so that $\tilde\alpha^i_n$ intersects $\tilde\alpha_j$ and $|i-j|$ is minimal. Then $|i-j|=2$. In fact, if, say, $i\geq j+4$ then $\tilde\alpha^{i+2}_n$ would intersect $\tilde\alpha_{j'}$ for $j'=i$ or $i+2< j'\leq j$, as suggested in Figure \ref{fig:disc_diagram}. More precisely, the intersection points of $\tilde\alpha^i_n$ with $\tilde\alpha_i$ and $\tilde\alpha_j$ lie in the same connected component of $\tilde Y-\tilde\alpha^{i+2}_n$ because distinct lifts of $\alpha_n$ do not intersect, while the endpoints of $\tilde\alpha_{i+2}$ lie in different connected components.  

 \begin{figure}[h]
  \includegraphics[scale=0.7]{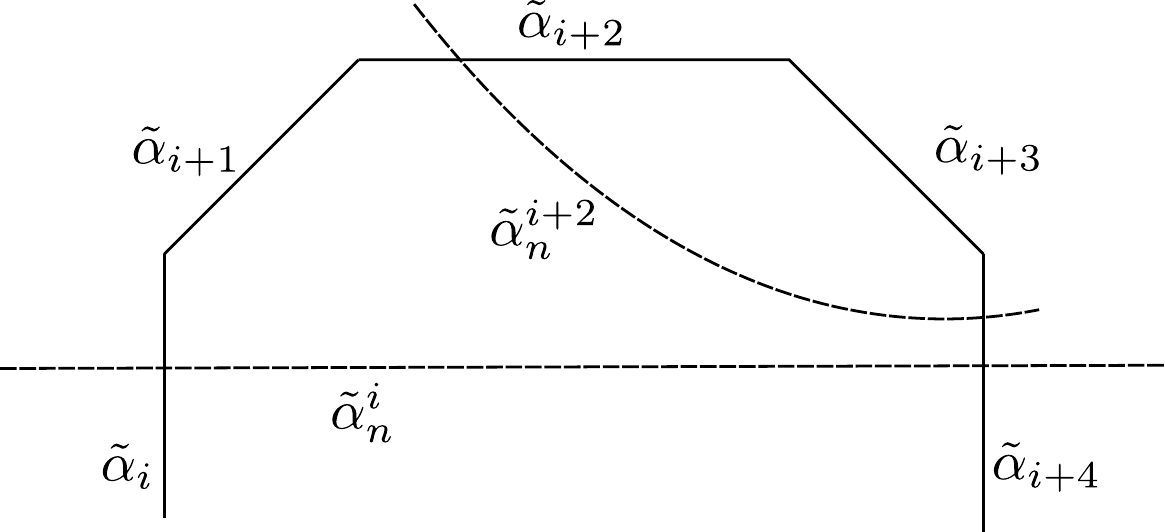}
  \caption{$\tilde\alpha^{i+2}_n$ cannot cross $\tilde\alpha^{i}_n$ (here $j=i+4$).}\label{fig:disc_diagram}
 \end{figure}

 Now, it is readily seen that all lifts of $\alpha_m$, $m$ odd, that intersect $\tilde\alpha_{i+1}$ in its interior also intersect $\tilde\alpha^i_n$ in its interior, see Figure \ref{fig:disc_diagram_2}. In fact, for any such lift $\tilde\alpha$, the intersection points of $\tilde\alpha^i_n$ with $\alpha_i$ and $\alpha_{i+2}$ each lie in the same connected component of $\tilde Y-\tilde\alpha$ as one of the endpoints of $\tilde\alpha_{i+1}$. 
 
  \begin{figure}[h]
  \includegraphics[scale=0.7]{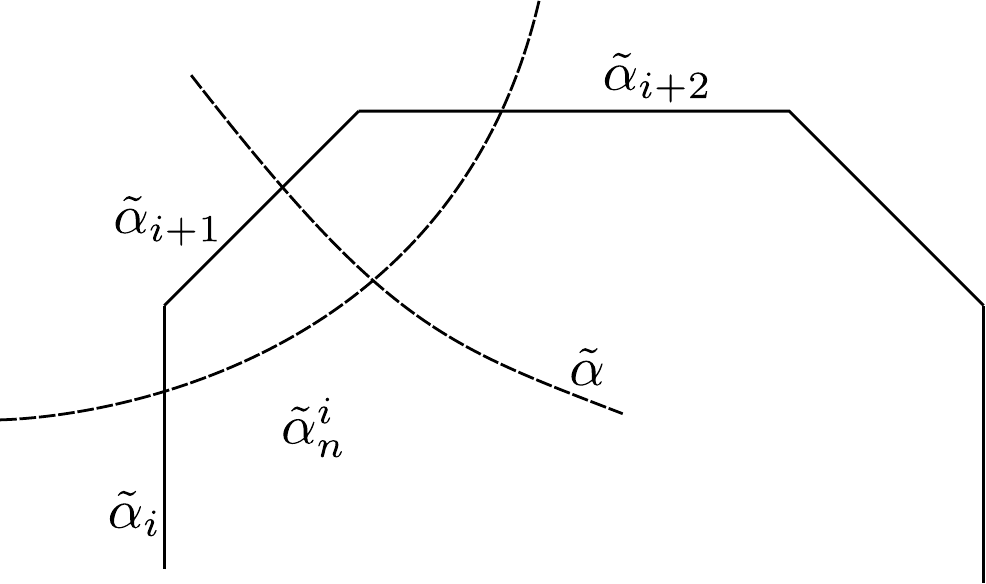}
  \caption{$\tilde\alpha$ is a lift of $\alpha_m$ for some odd $m$. If it crosses $\tilde\alpha_{i+1}$, it must also cross $\tilde\alpha^i_n$.}\label{fig:disc_diagram_2}
 \end{figure}

 Moreover, $\tilde\alpha^i_n$ also intersects in its interior two more lifts of some $\alpha_m$, namely $\tilde\alpha_i$ and $\tilde\alpha_{i+2}$. From this we deduce that the interior of $\alpha_{i+1}$ has fewer intersections with $\bigcup_{m\, odd} \alpha_m$ than the interior of $\alpha_n$, a contradiction.
\end{proof}

\subsection{Digression: Subsurface projection of discs}

We point out that Lemma \ref{lem:concat_fill_arcs}, besides being a key point in the proof of Theorem \ref{thm:steady_incompressible} below, also gives a significantly simpler proof of a useful lemma about subsurface projection of discs originally due to Masur-Schleimer \cite{MasurSchleimer} (which we need later). In fact, we improve the constants given by Masur-Schleimer (in this paper we measure distances between projection sets rather than diameters as in \cite{MasurSchleimer}; using the conventions of \cite{MasurSchleimer} we should replace ``1'' by ``3'' in both conclusions below).

In the statement we use the notions of arc graph $\mathcal A(X)$ of a surface with boundary $X$, that is defined similarly to the arc and curve graph using arcs only. Distances in the arc graph can be much larger than corresponding distances in the arc and curve graph, though, so the statement in terms of the arc graph is more refined than the corresponding statement in terms of the arc and curve graph.

\begin{cor}\label{cor:MasurSchleimer}(cfr. \cite[Lemma 12.20]{MasurSchleimer})
 Let $F$ be a compact surface with boundary (orientable or non-orientable) and let the handlebody $H$ be the orientable $[0,1]$--bundle of $F$. Then for every essential curve $d$ of $\partial H$ that bounds a disc of $H$ the following holds.
 \begin{itemize}
  \item If $F$ is orientable, and hence $H=F\times[0,1]$, let $X=F\times\{0\},Y=F\times\{1\}$ and let $\tau:X\to Y$ be the involution that switches the endpoints of the fibers. Then
  $$d_{\mathcal A(X)}(\pi_X(d),\tau(\pi_Y(d)))\leq 1.$$
  \item If $F$ is non-orientable, let $X=\partial H-(\partial X\times(0,1))$ and let $\tau:X\to X$ be the involution that switches the endpoints of the fibers. Then
  $$d_{\mathcal A(X)}(\pi_X(d),\tau(\pi_X(d)))\leq 1.$$
  Moreover, $\pi_X(d)$ lies within distance $1$ in $\mathcal{AC}(X)$ from a multicurve fixed by $\tau$.
  %thin distance $1$ of amu of $\tau$ in $\mathcal {AC}(X)$.
 \end{itemize}
\end{cor}

\begin{proof}
 Applying an isotopy, me can make sure that $d$ consists of a union of arcs each of which is either an essential arc of $X$ or $Y$ or a fiber over a boundary point of $F$. We can then disregard the arcs of the second type and have a sequence $\alpha'_1,\dots,\alpha'_k$ of arcs alternately in $X$ and $Y$ in the orientable case, and just in $X$ in the non-orientable case. Now, the arcs $\alpha_i=\tau^i(\alpha'_i)$ concatenate to form a homotopically trivial loop (in both cases, the inclusion of $X$ in $H$ is $\pi_1$--injective). By Lemma \ref{lem:concat_fill_arcs}, we get that some $\alpha_i$ for $i$ odd needs to be disjoint from $\alpha_j$ with $j$ even. This translates into the conditions described in the statement of the corollary.
 
  To get the conclusion about the fixed multicurve, consider some $\alpha'_i$ which is disjoint from some $\tau(\alpha_j')$. Then the subsurface $Z$ filled by $\alpha'_i$ and $\tau(\alpha'_i)$ is not the whole surface, and the set of essential curves in $\partial Z$ form the required multicurve.
\end{proof}

\section{Steady surfaces}

\subsection{Heegaard splittings}\label{subsec:heegard}
We denote by $H_g$ the oriented handlebody of genus $g$, and we identify $\Sigma_g=\partial H_g$. A \emph{Heegaard splitting} $\calH(\phi)$, where $\phi:\partial H_g\to\Sigma_g$ is a homeomorphism, is the $3$--manifold denote $M(\phi)$ obtained gluing two copies $H^0_g,H^1_g$ of $H_g$ to the boundary components $\Sigma_g\times\{0\}, \Sigma_g\times\{1\}$ of $\Sigma_g\times[0,1]$ using, respectively, the identity $\partial H^0_g\to\Sigma_g$ and $\phi:\partial H^1_g\to\Sigma_g$.\footnote{This always yields an orientable manifold; in order to obtain a 3-manifold and an orientation on it, we could insist that $\phi$ is orientation-preserving and reverse the orientation of one of the handlebodies.}  We can then define the disc set $\calD_i$, for $i=0,1$, as the set of all isotopy classes of curves on $\Sigma_g$ that bound a disc in the handlebody attached to $\Sigma_g\times\{i\}$.

 We say that the \emph{Heegaard distance} of $\calH(\phi)$ is the distance in the curve graph of $\Sigma_g$ between the disc sets.
% 
% \begin{lemma}\label{lem:steadyexists}
%  If $\calD_0,\calD_1$ are the disc sets associated to a Heegaard splitting of genus $g\geq 2$ and Heegaard distance at least $2$, then there exists a $(\calD_0,\calD_1)$--steady path.
% \end{lemma}
% 
% \begin{proof}
%  By \cite[Lemma 4.5]{MasurMinsky:II} there exists a \emph{steady geodesic} connecting some $d\in\calD_0$, $d'\in\calD_1$ that lie at minimal distance, so that in particular there exists a sequence (with $t_0=\{d\}$, $t_n=\{d'\}$) that satisfies all conditions in Definition \ref{defn:steadypath} except possibly for item \ref{item:maximal}. In order to arrange item \ref{item:maximal}, we just have to add to $t_0$ curves in $\calD_0$ disjoint from $t_1$, and similarly for $t_n$. (Notice that $t_1\neq t_n$, so the curves can be added ``independently''.)
% \end{proof}

\begin{defn}\label{defn:steadysurface}
 Let $\calH(\phi)$ be a Heegaard splitting with Heegaard distance at least $2$.
 
 A $d$--\emph{steady surface} in $M=M(\phi)$ is any embedded surface $S$ in $M$ constructed in the following way. Let $t_0,\dots,t_n$ be a $(\calD_0,\calD_1,d)$--steady path.
 %, where $n$ is the Heegaard distance of $\calH(g,\phi_0,\phi_1)$, be any steady geodesic from $\calD_0$ to $\calD_1$.
 %so that $t_0,t_n$ are both just a single curve.
 Choose open regular neighborhoods $N(t_i)$ of the multicurves $t_i$, with $\overline N(t_i)\cap \overline N(t_{i+1})=\emptyset$. Finally, let $S$ be the union of
 \begin{itemize}
  \item a union of disjoint discs in $H_g^0$ (resp. $H_g^1$) with boundary $\partial \overline N(t_0)$ (resp. $\partial \overline N(t_n)$),
  \item surfaces $S_i=S'_i\times\{i/(n+1)\}$, where $S'_i=\Sigma_g\setminus \Big(N(t_i)\cup N(t_{i-1})\Big)$, for $i=1,\dots,n$,
  \item the unions of annuli $A_i=\partial \overline N(t_i)\times [i/(n+1),(i+1)/(n+1)]$, for $i=0,\dots,n-1$.
 \end{itemize}
 
\end{defn}

We remark that steady surfaces are orientable.

It is easy to give conditions for a steady surface to have a connected component which is not a sphere, for example:

\begin{lemma}\label{lem:Euler}
 Suppose that $t_0,\dots,t_n$ is a steady path for the surface $\Sigma_g$ defining the steady surface surface $S$. If either $g\geq 3, n\geq3$ or $g\geq 2,n\geq 4$, and $t_1,t_{n-1}$ consist of a single curve, then $S$ has a connected component which is not a sphere. 
\end{lemma}

\begin{proof}
 We use the notation of Definition \ref{defn:steadysurface}. Consider the subsurface $S'$ consisting of the union of all $S_i$ with $i\neq 1,n$ and the annuli $A_i$ for $i=2,\dots,n-2$ (no annuli if $n=3$). The Euler characteristic of $S'$ is $ (2-2g)(n-2)$, and $S'$ has $4$ boundary components. Under both sets of assumptions on $(g,n)$, we can conclude that $S'$ has positive genus, and hence one connected component of $S$ is not a sphere.
\end{proof}

Finally, we prove the key result to construct Haken manifolds from Heegaard splittings.

\begin{thm}\label{thm:steady_incompressible}
Let $\calH(\phi)$ be a Heegaard splitting with $g\geq 2$ of distance at least $2$. If $S$ is a $5$--steady surface in $M=M(\phi)$, then $S$ does not admit a compression disc.
%not admit a compression disc. If the Heegaard distance is at least $2$ then one of the components of the steady surface is not a sphere and hence it is incompressible.
%Let $\calH(g,\phi_0,\phi_1)$ be a Heegaard splitting and let $S$ be any steady surface. If the Heegaard distance is at least $2$, then $S$ does not admit a compression disc. If the Heegaard distance is at least $3$, then $S$ has a connected component that is not a sphere. 
\end{thm}

\begin{proof}
%We first show that $S$ admits no compression disc, and then that one of its components is not a sphere.
 We use the notation of Definition \ref{defn:steadysurface}. Suppose by contradiction that the steady surface $S$ admits a compression disc $D$. We now isotope $D$ in a suitable normal form in a few steps.
 
 First of all, applying an isotopy we can assume that the boundary of $D$ is either
 \begin{enumerate}
  \item a simple loop contained in some $S_i$, or\label{item:loop}
  \item a union of essential arcs in the $S_i$, that we call \emph{vertical arcs}, and arcs in the $A_i$ of the form $\{p\}\times  [i/(n+1),(i+1)/(n+1)]$, that we call \emph{horizontal arcs}.\label{item:arcs}
 \end{enumerate}
 
 This is just because we can remove inessential arcs in the various $\partial D\cap S_i$ and $\partial D\cap A_i$ starting from innermost ones.
 
 Moreover, we will consider:
 \begin{enumerate}
  \item simple loops in $\mathring D$ which are connected components of the intersection of $\mathring D$ and $\overline{\Sigma}_g=\bigcup \Sigma_g\times\{i/(n+1)\}$. We call these \emph{intersection loops}.
  \item simple arcs in $D$ whose interior is a connected component of the intersection of $\mathring D$ and $\overline{\Sigma}_g$. We call these \emph{intersection arcs}.
 \end{enumerate}

 We can assume that there are finitely many intersection loops and intersection arcs.
 %Notice that such loops and arcs need not be properly embedded in $D$.
 
%  \par\medskip
 
 \emph{Ruling out intersection loops.} We now argue that we can isotope $D$ to remove intersection loops and that case \ref{item:loop} does not occur. Consider an innermost intersection loop $\ell$ and suppose that it is contained in $S_i$.  If $\ell$ bounds a disc in $S_i$ then a simple surgery arguments allows us to replace $D$ by a disc that has fewer loops in the intersection with $\overline\Sigma_g$, hence we can assume that $\ell$ is essential in $\Sigma_g\times\{i/(n+1)\}$ (we are not ruling out that it is parallel into the boundary of $S_i$, for now). Notice that the subdisc of $D$ bounded by $\ell$ does not intersect one of the handlebodies and hence, when identifying $\Sigma_g\times\{i/(n+1)\}$ with $\Sigma_g$, we have $\ell\in \calD_0\cup \calD_1$. Definition \ref{defn:steadypath}.\ref{item:filling} rules out that $\ell$ is contained in $S_i$ for $i\neq 1,n$, so that $\ell$ is contained in $S_1$ or $S_n$. But then Definition \ref{defn:steadypath}.\ref{item:maximal} implies that $\ell$ is parallel to a component of $t_0\times\{1/(n+1)\}$ or $t_n\times\{n/(n+1)\}$, which bound discs in $S$. Hence, we can once again replace $D$ with a disc that has fewer loops in the intersection with $\overline\Sigma_g$. 
 
 We can then go on and remove all loops in $\mathring D\cap \overline{\Sigma}_g$, and finally a very similar argument proves that case \ref{item:loop} does not occur because otherwise  $\partial D$ would not be essential in $S$.
 
 %\par\medskip
 
\emph{Cutting up $D$.} We now have that there are no intersection loops, just intersection arcs.
%$D\cap\overline{\Sigma}_g$ consists of finitely many arcs, and we will refer to the arcs in $D\cap\overline{\Sigma}_g$ not entirely contained in $\partial D$ as intersection arcs.
The intersection arcs subdivide $D$ into finitely many closed (polygonal) regions. The Euler characteristic of $D$ equals the number of such regions minus the number of intersection arcs. We are now going to argue that each region contains at least two intersection arcs, leading to a contradiction.

In fact, suppose that a region $R$ contains only one intersection arc $\alpha$, say contained in $S_i$. The closure of $\partial R-\alpha$ is an arc $\beta$ contained in $\partial D$, and hence it is a concatenation of (alternately) horizontal and vertical arcs. Moreover, the vertical arcs are all contained in either $S_i\cup S_{i+1}$ or $S_{i-1}\cup S_i$, for otherwise the interior of $R$ would have to intersect either $S_{i-1}$ or $S_{i+1}$. Since $R$ is contained in $\Sigma_g$ times an interval, we can then project $\partial R$ to the factor $\Sigma_g$, and obtain a concatenation of paths as described in Lemma \ref{lem:concat_fill_arcs} (notice that an arc in $\Sigma_g-N(t_i)$ disjoint from $t_{i-1}$ intersects any arc in $\Sigma_g-N(t_i)$ that intersects $t_{i+1}$ at most once since $d_{\mathcal{AC}(\Sigma_g-N(t_i))}(t_{i-1},t_{i+1})\geq5$). However, such concatenation is homotopically trivial because we can also project $R$, a contradiction.
\end{proof}

\section{Construction of Haken manifolds}
\label{sec:constr}

 \subsection{Fixing gluing maps}
 
 Fix a genus $g\geq 3$. Then the handlebody $H_g$ of genus $g$ can be identified with the product $F\times [0,1]$, where $F$ is a sphere with at least $4$ discs removed. We denote by $\sigma_0$, $\sigma_1$ the core curves of two connected components of $\partial F\times[0,1]$. We claim that there exists a Heegaard splitting $H(\iota=\iota_1\circ\iota_2)$ of the sphere $S^3$ so that $\iota_i\in Stab(\sigma_i)$. In fact, there is a curve $c$ on $\partial H_g$ that bounds a disc and separates $\sigma_1$ from $\sigma_2$; just consider the product of an arc in $F$ that separates the components of $\partial F$ corresponding to $\sigma_1,\sigma_2$, and take the product with $[0,1]$. Now, the usual gluing map that exchanges meridian and longitudes can be written as a product of two homeomorphisms that each restrict to the identity on one component of the complement of $c$, as required.
 
 Denote by $\calD$ the disc set of $H_g$, so that the disc sets associated to the Heegaard splitting are $\calD$ and $\iota(\calD)$.
 
%  \subsubsection{Genus 2} In genus 2 we need a slightly different construction because there is only a limited amount of ``space''. The handlebody $H_2$ of genus $2$ is homeomorphic to the orientable $[0,1]$--bundle of the the M\"obius strip with a disc $D$ removed, that we call $F$. We denote by $\sigma_0$, the core curves of the annulus that fibers over $\partial \overline{D}$. The curve $\sigma_0$ is depicted in Figure \ref{fig:genus2} in the standard picture of the handlebody (the the orientable $[0,1]$--bundle of the the M\"obius strip is a solid torus, and the the orientable $[0,1]$--bundle of $F$ is obtained from $F$ by ``drilling a hole''). We take $\sigma_1$ to be the other curve in the picture, which admits a similar description as $\sigma_0$ in terms of another bundle. 
%  
%  Everything else is as above: There exists a Heegaard splitting $H(\iota=\iota_1\circ\iota_2)$ of the sphere $S^3$ so that $\iota_i\in Stab(\sigma_i)$, and we denote by $\calD$ the disc set of $H_2$.
 
 We fix the data described in this subsection from now on.
 
%  \begin{figure}[h]
%   \includegraphics[scale=0.6]{genus2}
%   \caption{}\label{fig:genus2}
%  \end{figure}

 \subsection{Constructing large partial pseudo-Anosovs}

 Let $\mathcal K$ be the subgroup of $MCG(\Sigma_g)$ generated by Dehn twists around separating curves. By \cite{Morita}, there is a homomorphism $J:\mathcal K\to \mathbb Z$ so that, for all $\phi\in\mathcal K$, the Casson invariant of $H(\iota\circ\psi)$ is $J(\psi)$.
 
 \begin{lemma}\label{lem:product_twists}
 For $i=0,1$ the following holds. For every $L$ there exists $\phi_i\in \mathcal K\cap Stab(\sigma_i)$ so that $J(\phi_i)=1$, $\phi_i$ has a geodesic axis in $\mathcal {AC}(\Sigma_g- N(\sigma_i))$, and the translation distance of $\phi$ is at least $L$.
 % Let $\calH(g,\iota,id)$ be the standard Heegaard splitting of genus $g$ of $S^3$, with corresponding disc sets $\calD_0,\calD_1$. Set $\calD=\calD_0\cup\calD_1$. Then for every $d$ there exist subsurfaces $Y_0,Y_1$ and $\phi_i\in Stab(Y_i)$ satisfying the hypotheses of Lemma \ref{lem:construct_steady} and so that, furthermore, $\phi_i\in \mathcal K$ and the Casson invariant of $\calH(g,\iota,\phi_i)$ is $1$.
 \end{lemma}

 \begin{proof}
 It is proven in \cite[Lemma 5]{LubotzkyMaherWu} that there is a Dehn twist $\tau_{c}$ around a separating curve so that $J(\tau_c)=1$. In fact, they obtain $c$ as follows. First, they embed a trefoil knot on a Heegaard surface of genus $2$ in $S^3$ (see \cite[Figure 2]{LubotzkyMaherWu}), and then they stabilise the Heegaard splitting, using stabilisations disjoint from such embedding, to get a curve $c$ on a Heegaard surface of the required genus. Up to isotopy, $\sigma_i$ is contained in one of the solid tori that one connect-sums to the Heegaard surface of genus 2 to stabilise (recall that we are dealing with the case of genus at least 3). In particular, one can ensure that $c$ is contained in $\Sigma_g-N(\sigma_i)$.
 
 Set $c_1=c$ and let $c_2$ be of the form $k(c_1)$ for some $k\in\mathcal K$ and with the property that $c_1$ and $c_2$ are at least $L+2$ apart from each other in $\mathcal {AC}(\Sigma_g- N(\sigma_i))$.
  
  We now consider any product $\phi_i$ of sufficiently large powers of the $\tau_{c_j}$. In order to construct a geodesic axis for $\phi_i$, there is a standard procedure (see the proof of \cite[Proposition 3.3]{Mangahas}, and in particular \cite[Lemma 3.10]{Mangahas}): One starts from geodesics connecting the $c_j$, takes subgeodesics connecting curves disjoint from the $c_j$, starts ``rotating'' these using the $\tau_{c_j}$ and takes concatenations. It is now easy to use the Bounded Geodesic Image Theorem, similarly to Lemma \ref{lem:concat_large_angle}, to show that such concatenation is a geodesic line, and that the translation distance is at least $L$ (we have not recalled the definition of the arc graph of an annulus, but the only fact about it that is needed for this construction is that the corresponding Dehn twist acts with positive translation length).
  
  Finally, choosing the powers of the $\tau_{c_j}$ suitably, we can further ensure $J(\phi_i)=1$, as required.
 \end{proof}

\subsection{Moving $\sigma_i$ off the disc set}

\begin{lemma}\label{lem:rotate_off_disc_set}
 For $i=0,1$ and for every $d$, there exists $\phi_i\in \mathcal K\cap Stab(\sigma_i)$ with $J(\phi_i)=1$ so that for every integer $k\neq 0$ we have
 $$d_{\mathcal {AC}(\Sigma_g-N(\sigma_i))}(\calD,\phi^k_i(\sigma_{i+1}))\geq d,$$
 $$d_{\mathcal {AC}(\Sigma_g-N(\sigma_i))}(\calD,\phi^k_i\iota_i(\sigma_{i+1}))\geq d.$$
 %$$d_{\mathcal {AC}(\Sigma_g-N(\sigma_i))}(\calE,\phi^k_i(\sigma_{i+1}))\geq d.$$
\end{lemma}

\begin{proof}
We give the proof for genus at least $3$ first.

Let $L,R$ be large enough constants to be determined later.

 Let $\phi_i$ be as in Lemma \ref{lem:product_twists}. Let $X=F\times\{0\}, Y=F\times\{1\}$. We can conjugate $\phi_i$ to ensure every curve along the axis $\gamma$ of $\phi_i$ cuts $\partial X$ and $\partial Y$. By conjugating $\phi_i$ by a large pseudo-Anosov of $X$, and keeping into account that the entire axis of $\gamma$ has bounded projection onto $\mathcal{AC}(X)$ and $\mathcal{AC}(Y)$, we can then ensure that, identifying $X,Y$ with $F$, for every curve $c$ on $\gamma$ we have $d_{\mathcal {AC}(F)}(\pi_X(c),\pi_Y(c))\geq R$.
 
 Notice that $\phi_i^k(\sigma_{i+1})$ has closest point projection to $\gamma$ far away from that of $\partial X$. If there was $c\in \calD$ so that $d_{\mathcal C(\Sigma_g-N(\sigma_i))}(c,\phi_i^k(\sigma_{i+1}))$ is small, then, by a simple Gromov-hyperbolicity argument, we would have a geodesic from $\pi_{\Sigma_g-N(\sigma_i))}(c)$ to $\gamma$ that stays far from $\partial X$ and $\partial Y$. Hence, $d_{\mathcal {AC}(F)}(\pi_X(c),\pi_Y(c))$ would be large, contradicting Corollary \ref{cor:MasurSchleimer}. A similar argument holds for $\phi_i^k(\iota_i(\sigma_{i+1}))$, which lies within uniformly bounded distance of $\phi_i^k(\sigma_{i+1})$.
 
\begin{figure}[h]
 \includegraphics[scale=0.7]{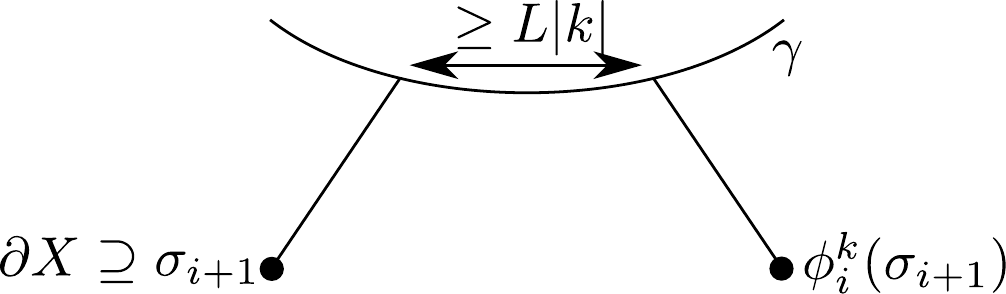}
 \caption{Picture in $\mathcal C(\Sigma_g-N(\sigma_i))$.}
\end{figure}

 For genus $2$, the proof is similar. In this case we let $X$ be the double cover of $F$ contained in $\partial H_2$. We can conjugate $\phi_i$ so that its axis has bounded projection onto $\mathcal{AC}(X)$. Furthermore, since the M\"obius strip with one disc removed has bounded curve graph, we can make sure that such projection lies far away from multicurves fixed by $\tau$, where $\tau$ is the involution described in Corollary \ref{cor:MasurSchleimer} (in the notation of the corollary, there is a natural correspondence between multicurves fixed by $\tau$ and multicurves of $F$). Finally, we can show that if the projection to $\mathcal{AC}(\Sigma_g-N(\sigma_i))$ of some $d\in\calD$ was close to either of $\phi_i^k(\sigma_{i+1})$ or $\phi_i^k(\iota_i\sigma_{i+1})$, then its projection to $X$ would be far away from the fixed set of $\tau$, contradicting Corollary \ref{cor:MasurSchleimer}.
 %Since the M\"obius strip with one disc removed has bounded arc and curve graph, Corollary \ref{cor:MasurSchleimer} implies that 
 \end{proof}

 \subsection{Constructing steady paths}
 
 \begin{lemma}\label{lem:product}
 Fix a large enough $d$ and let $\phi_i$ be as in Lemma \ref{lem:rotate_off_disc_set}.
 
  Let $n\geq 2$ and let $k_1,\dots,k_n$ be nonzero integers. Let $\psi$ be the product, with indices modulo 2,
  $$\psi=(\phi^{k_1}_1\iota_1)(\phi^{k_2}_2\iota_2)\prod_{i=3}^{n}\phi^{k_i}_i.$$
  Then there exists a $(\calD,\psi(\calD),d)$--steady path $t_0,\dots,t_n$ with $t_i$ consisting of a single curve for $i\neq 0,n$. Moreover, $d_{\calC(\Sigma_g)}(\calD,\psi(\calD))=n+1$.
 \end{lemma}

\begin{proof}
Let $\phi'_1=\phi^{k_1}_1\iota_1,\phi'_2=\phi^{k_2}_2\iota_2$ and $\phi'_i=\phi^{k_i}_i$ for $i\geq 3$.

For $k\geq 3$, let $\psi_k=\prod_{i=1}^{k}\phi'_i$. Let $t_i=\psi_{i-1}\sigma_i=\psi_i\sigma_i$ for $i=1,\dots,n$. Also, let $t_0\subseteq \calD$ (resp. $t_{n+1}\subseteq \psi(\calD)$) be a maximal collection of pairwise disjoint curves disjoint from $t_1$ (resp. $t_n$).

All conditions of Definition \ref{defn:steadypath} except for item \ref{item:filling} can be easily verified directly. To verify item \ref{item:filling} we just need to observe that, provided $d$ is large enough, for every $d\in\calD$, every geodesic from $d$ to $t_i$, $i\geq 2$, contains $t_1,\dots,t_{i-1}$ by Lemma \ref{lem:concat_large_angle}, and similarly for $d'\in\psi(\calD)$. Also, any geodesic from $\calD$ to $\psi(\calD)$ contains $t_1,\dots,t_n$, proving $d_{\calC(\Sigma_g)}(\calD,\psi(\calD))=n+1$ (since $d_{\calC(\Sigma_g)}(\calD,\psi(\calD))\leq n+1$ because $t_0,\dots,t_{n+1}$ form a path).
\end{proof}

\subsection{Proof of Theorem \ref{thm:manyhaken}}

Fix $g,n$ and $k$ as in the statement of the theorem. By Lemma \ref{lem:product}, we can construct a Heegaard splitting $H(\psi)$ so that the two disc sets $\calD_0,\calD_1$ in $\calC(\Sigma_g)$ are at distance exactly $n$, and they are connected by a $(\calD_0,\calD_1,5)$--steady path with $t_1,t_{n-1}$ consisting of a single curve. Hence the resulting manifold $M(\psi)$ is Haken by Theorem \ref{thm:steady_incompressible} and Lemma \ref{lem:Euler} (notice that if a possibly disconnected surface does not admit a compression disc then none of its components do). Moreover $\psi=\iota \phi$ for some $\phi\in \mathcal K$ (since $\mathcal K$ is normal), so that $M(\psi)$ is an integer homology sphere and by choosing the exponents $k_i$ in Lemma \ref{lem:product}, we can make sure that $J(\phi)=k$, i.e. that the Casson invariant of $M(\psi)$ is $k$.

Since the Heegaard splitting has distance at least 3, the resulting manifold is hyperbolic by results in \cite{Hempel:Heegaard} and Thurston's hyperbolisation.\qed

  \bibliographystyle{alpha}
 \bibliography{biblio.bib}

\end{document}